\documentclass[a4paper,11pt,oneside,headings=small,bibliography=totoc,abstract=true]{scrartcl}
\usepackage[linktocpage, colorlinks]{hyperref} %pdftex,
	\usepackage{color}
 	\definecolor{darkred}{rgb}{0.5,0,0}
 	\definecolor{darkgreen}{rgb}{0,0.5,0}
 	\definecolor{darkblue}{rgb}{0,0,0.5} 	\hypersetup{colorlinks,linkcolor=darkblue,filecolor=darkgreen,urlcolor=darkred,citecolor=darkblue}
\hypersetup{
pdftitle={Observability and null-controllability for parabolic equations in \texorpdfstring{$L_p$}{Lp}-spaces},
pdfauthor={Clemens Bombach, Dennis Gallaun, Christian Seifert, Martin Tautenhahn},
}
\usepackage{amssymb,amsmath,amsfonts,amsthm,enumerate,mathtools}
    \allowdisplaybreaks
\usepackage[english]{babel}
\usepackage[noblocks]{authblk}

\newcommand{\thickset}{E}
\newcommand{\ii}{\mathrm{i}}
\newcommand{\drm}{\mathrm{d}}
\newcommand{\euler}{\mathrm{e}}
\newcommand{\Z}{\mathbb{Z}}

\newcommand{\N}{\mathbb{N}}
\newcommand{\R}{\mathbb{R}}
\newcommand{\C}{\mathbb{C}}

\newcommand{\F}{\mathcal{F}}

\newcommand{\1}{\mathbf{1}}
\newcommand{\from}{\colon}
\newcommand{\re}{\operatorname{Re}}

\newcommand{\supp}{\operatorname{supp}}
\newcommand{\id}{\operatorname{Id}}
\newcommand{\ran}{\operatorname{Ran}}
\DeclareMathOperator*{\esssup}{ess\,sup}
\newcommand{\norm}[1]{\left\lVert#1\right\rVert}
\newcommand{\abs}[1]{\left\lvert#1\right\rvert}
\newtheorem{Theorem}{Theorem}[section]

\newtheorem{Proposition}[Theorem]{Proposition}
\theoremstyle{definition}

\newtheorem{Definition}[Theorem]{Definition}
\theoremstyle{remark}
\newtheorem{Remark}[Theorem]{Remark}

\title{Observability and null-controllability for parabolic equations in $L_p$-spaces}
\author{Clemens Bombach}
\affil{Technische Universit\"at Chemnitz, Fakult\"at f\"ur Mathematik, 09107 Chemnitz, Germany}
\author{Dennis Gallaun}
\affil{Technische Universit\"at Hamburg, Institut f\"ur Mathematik, 
21073 Hamburg, Germany}
\author[2,3]{Christian Seifert}
\affil{Technische Universit\"at Clausthal, Institut f\"ur Mathematik, 
38678 Clausthal-Zellerfeld, Germany}
\author[4]{Martin Tautenhahn}
\affil{Universit\"at Leipzig,
Mathematisches Institut,
04109 Leipzig,
Germany
}
\date{\vspace{-7ex}}
\begin{document}
\maketitle
\begin{abstract}
We study (cost-uniform approximate) null-controllability of parabolic equations in $L_p(\R^d)$ and provide explicit bounds on the control cost. 
In particular, we consider systems of the form
$\dot{x}(t) = -A_px(t) + \1_\thickset u(t)$, $x(0) = x_0\in L_p (\R^d)$, with interior control on a so-called thick set $\thickset \subset \R^d$, where $p\in [1,\infty)$, and where $A$ is an elliptic operator of order $m \in \N$ in $L_p(\mathbb{R}^d)$. 
We prove null-controllability of this system via duality and a sufficient condition for observability. This condition is given by an uncertainty principle and a dissipation estimate. 
Our result unifies and generalizes earlier results obtained in the context of Hilbert and Banach spaces. In particular, our result applies to the case $p=1$.
 %We show that a generalized uncertainty principle and a dissipation estimate imply (approximate) null-controllability. Our result unifies and generalizes earlier results obtained in the context of Hilbert and Banach spaces. In particular, we do not assume reflexivity of the underlying Banach space, thus allowing to apply our result, e.g., to $L_1$-spaces.
%%\par
%As an application we consider parabolic equations of the form $\dot{x}(t) = -Ax(t) + \1_\thickset u(t)$, $x(0) = x_0\in L_p (\R^d)$, with interior control on a so-called thick set $\thickset \subset \R^d$, where $p\in [1,\infty)$, and where $A$ is an elliptic operator of order $m \in \N$ in $L_p(\mathbb{R}^d)$. In particular, our result applies to the case $p=1$.
\\[1ex]
\\[1ex]
\end{abstract}
\section{Introduction}

We consider parabolic control systems on $L_p(\R^d)$, $p\in [1,\infty)$, of the form
\begin{equation}
  \label{eq:system-intro}
  \dot{x}(t) = -A_p x(t) + \1_\thickset u(t),\quad t\in (0,T],\quad x(0) = x_0\in L_p(\R^d),
\end{equation}
where $-A_p$ is a strongly elliptic differential operator of order $m \in \N$ with constant coefficients, $\1_\thickset \from L_p (\thickset)\to L_p(\R^d)$ is the embedding from a measurable set $\thickset \subset \R^d$ to $\R^d$,  $T>0$, and where $u \in L_r ((0,T);L_p (\thickset))$ with some $r \in [1,\infty]$. Hence, the influence of the control function $u$ is restricted to the subset $\thickset$. Note that we allow for lower order terms in the strongly elliptic differential operator. 
The focus of this paper is laid on null-controllability, that is, for any initial condition $x_0\in L_p(\R^d)$ there is a control function $u\in L_r ((0,T);L_p (\thickset))$ such that the mild solution of \eqref{eq:system-intro} at time $T$ equals zero. We will also be concerned with the notion of cost-uniform approximate null-controllability (or approximate null-controllability with uniformly bounded controls), which means that there exists $C>0$ such that for all $\varepsilon > 0$ and all $x_0\in L_p(\R^d)$ with $\norm{x_0}_{L_p(\R^d)}\leq 1$ we can find a control function $u \in L_r ((0,T);L_p (\thickset))$ with $\norm{u}_{L_r ((0,T);L_p (\thickset))}\leq C$ such that the mild solution of \eqref{eq:system-intro} at time $T$ has norm smaller than $\varepsilon$; in reflexive spaces, these two notions agree (see \cite{Carja-88}). By linearity, (cost-uniform approximate) null-controllability implies that any target state in the range $\ran(\euler^{-TA_p})$ of the semigroup generated by $-A_p$ can be reached (up to an error $\varepsilon$) within time $T$. 
\par
We will show in Theorem \ref{thm:null-control} that if $\thickset$ is a so-called thick set, then the system is cost-uniform approximately null-controllable for $p=1$ and null-controllable if $p\in (1,\infty)$. 
Thus, we generalize the results of \cite{GallaunST-20} in two ways. On the one hand, we are able to deal with the case $p=1$ which is important for applications; just note that, e.g., the classical heat equation in $L_1$ yields interpretations of the solution as heat densities, while diffusing population models were also considered in $L_1$-spaces. On the other hand, we allow for lower order terms in the operator $-A_p$.
%Note that the case $p \in (1,\infty)$ is already covered by \cite{GallaunST-20}. However, our new proof unifies the cases $p = 1$ and $p \in (1,\infty)$. 
Moreover, we provide explicit upper bounds on the control cost, i.e.\ on the norm of the control function $u$ which steers the system (approximately) to zero at time $T$, which are explicit in terms of geometric properties of the thick set $\thickset$ and of the final time $T$. 
\par
Null-controllability for (linear) systems in $L_p(\Omega)$ with control function in $L_r$ has already been studied earlier in the literature, both in bounded as well as unbounded domains $\Omega\subseteq\R^d$. However, mostly the investigations are done in the Hilbert space context, i.e.\ $p=r=2$.
More precisely, for bounded domains $\Omega$, null-controllability has been considered in \cite{FattoriniR-71,LebeauR-95,FursikovI-96,ErvedozaZ-11,MartinRR-14} with various methods. Note that in many applications $-A_2$ is self-adjoint and thus spectral theoretic methods are applicable. Null-controllability for unbounded domains $\Omega$ has been studied in \cite{Teresa-97, CabanillasMZ-01, MicuZ-01a, MicuZ-01, CannarsaMV-04, Miller-05, KalimerisO-20}. 
%\todo{eingefuegter Satz}
The fact that thickness of $E$ is necessary and sufficient to obtain null-controllability for the heat equation in $L_2(\R^d)$ has been realized in \cite{EgidiV-18,WangWZZ-19}.
Turning away from the Hilbert space case, we mention \cite{GallaunST-20}, where $\Omega=\R^d$, $p\in (1,\infty)$ and $r\in[1,\infty]$.

%Systems for $p=1$: Population system with spatial diffusion: \cite{Rhandi-98,RhandiS-99} ($L_1$ norm gives the size of population) \todo{Beispiel für $L_1$; passt aber leider nicht ganz in unser Setting.}

\par
An equivalent formulation of cost-uniform approximate null-controllability is final-state observability of the dual system to \eqref{eq:system-intro}. This means that there exists a constant $C_{\mathrm{obs}} \geq 0$ such that for all $\varphi \in L_p(\R^d)'$ we have
\begin{equation*}
	\norm{S'_T \varphi}_{L_p(\R^d)'} \leq \begin{cases}
    C_{\mathrm{obs}} \left(\int_0^T \norm{(S'_t \varphi)|_\thickset}_{L_p(\thickset)'}^{r'} \drm t\right)^{1/r'} & \text{if } r'\in [1,\infty),\\
    C_{\mathrm{obs}}\esssup\limits_{t\in [0,T]} \norm{(S'_t \varphi)|_\thickset}_{L_p(\thickset)'} & \text{if } r'=\infty,
 \end{cases} 
\end{equation*}
where $(S_t)_{t\geq 0}$ is the $C_0$-semigroup generated by $-A_p$ and $r' \in [1,\infty]$ is such that $1/r + 1/r' = 1$.
This equivalence follows from Douglas' lemma, see \cite{Douglas-66} for Hilbert spaces, and \cite{Embry-73,DoleckiR-77,Harte-78,CurtainP-78,Carja-85,Carja-88,Forough-14} for Banach spaces.
% \par
% \todo{M: Nächsten Abschnitt weglassen?}One motivating example for the described setting is the classical heat equation, i.e.\ $A=-\Delta$ in $L_p(\R^d)$ with $p\in [1,\infty)$, and with interior control on the measurable set $\thickset \subset \R^d$. In this case $(S_t)_{t\geq0}$ is the Gau{\ss}-Weierstra{\ss} semigroup.
\par

%\todo{Absatz zu Y on v2 eingefuegt}
%\todo{possible generalizations to nonlinear systems}
% In an $L_2$-setting, the general strategy to prove null-controllability by establishing a final-state observability estimate has also been successfully applied to certain nonlinear problems, e.g.\ the non-linear heat equation; cf.\ e.g.\ \cite{Fernandez-Cara-97,CabanillasMZ-01,Gonzalez-BurgosT-07}. There, final-state observability is shown for the linearized system and is then combined with a suitable fixed point argument. We will pursue this question in the framework of Banach spaces in a future publication.
% %\todo{exact and approximate control to trajectories}
% Similarly, control to trajectories has been considered for particular systems in the $L_2$-setting only. We refer to \cite{BenabdallahCGT-10,LazarM-21} and references therein.
% %\todo{the way good null controls can be found}
% Numerical methods to compute or approximate the control functions have been considered for the heat equation (on bounded domains) in $L_2$, see e.g.\ \cite{MuenchZ-10}. For systems in (abstract) Hilbert spaces, one can either consult the algorithm obtained in \cite{LazarM-21}, or employ that control functions can be considered as orthogonal projections on an affine subspace, see \cite[Remark 2.6]{NakicTTV-20}. 

\par
In Section \ref{sec:application} we formulate our results on final-state observability in Theorem \ref{thm:obs}. Then the (cost-uniform approximate) null-controllability in Theorem \ref{thm:null-control} follows as a consequence of the above-mentioned duality.
The proof of Theorem \ref{thm:obs}, which can be found in Section \ref{sec:dissipation}, rests on an abstract observability estimate stated in the appendix. There, Theorem \ref{thm:spectral+diss-obs} provides a generalization of the abstract result in \cite{GallaunST-20} for not necessarily strongly continuous semigroups. In view of \cite{Lotz-85}, this is particularly important for final-state observability  in $L_\infty$ or, put differently, cost-uniform approximate null-controllability in $L_1$.
\par
The main strategy we follow to prove observability has first been described for the Hilbert space case (i.e.\ $p=r=2$) in \cite{Miller-10} inspired by \cite{LebeauR-95,LebeauZ-98,JerisonL-99}, and further studied, e.g., in \cite{TenenbaumT-11,WangZ-17,BeauchardP-18,NakicTTV-20,Barcena-PetiscoZ-21}. It combines an (abstract) uncertainty principle or unique continuation estimate with a dissipation estimate to obtain the final-state observability via an iterative argument; cf.\ Theorem \ref{thm:spectral+diss-obs}. 
However, far less is known about its generalization to Banach spaces; to the best of our knowledge, we are only aware of \cite{GallaunST-20}, which applied the strategy for the case of strongly continuous semigroups.
To finally obtain Theorem \ref{thm:obs} we thus need to show the uncertainty principle and the dissipation estimate. While the uncertainty principle is a consequence of the Logvinenko--Sereda theorem (see Theorem \ref{Thm:Logvinenko-Sereda_Rd}), the dissipation estimate is shown via explicit estimates on the kernel of the semigroup; cf.\ Proposition \ref{Prop:Dissipation}. This enables us to cover the case $p\in\{1,\infty\}$ as well, in contrast to the interpolation technique used in \cite{GallaunST-20}.
\par
In Section \ref{sec:further_directions} we discuss further directions and developments, particularly focussing on control to trajectories, some nonlinear problems as well as the question of how good null controls can actually be found.
\section{Observability and Null-controllability in \texorpdfstring{$L_p$}{Lp}-Spaces}
\label{sec:application}

In order to formulate our main theorems, we review some basic facts from Fourier analysis. For details we refer, e.g., to the textbook \cite{Grafakos-14}. We denote by $\mathcal{S}(\R^d)$ the Schwartz space of rapidly decreasing functions, which is dense in $L_p(\R^d)$ for all $p \in[1,\infty)$. The space of tempered distributions, i.e.\ the topological dual space of $\mathcal{S}(\R^d)$, is denoted by $\mathcal{S}'(\R^d)$.
For $f\in \mathcal{S}(\R^d)$ let $\F f\from\R^d\to\C$ be the Fourier transform of $f$ defined by
\[\F f (\xi) := \int_{\R^d} f(x) \euler^{-\ii\xi\cdot x}\drm x.\]
Then $\F\from \mathcal{S}(\R^d)\to \mathcal{S}(\R^d)$ is bijective, continuous and has a continuous inverse, given by
\[\F^{-1} f(x) = \frac{1}{(2\pi)^d} \int_{\R^d} f(\xi) \euler^{\ii x\cdot \xi}\drm \xi\]
for all $f\in \mathcal{S}(\R^d)$. 
For $u\in \mathcal{S}'(\R^d)$ the Fourier transform is again denoted by $\F$ and is given by $(\F u)(\phi) = u(\F \phi)$ for $\phi\in \mathcal{S}(\R^d)$. By duality, the Fourier transform is bijective on $\mathcal{S}'(\R^d)$ as well.
\par
Let $m \in \N$ and
\begin{equation*}
 a(\xi) = \sum _{\abs{\alpha}_1 \leq m} a_\alpha\xi^\alpha, \quad \xi \in \R^d,
\end{equation*}
be a polynomial of degree $m$ with coefficients $a_\alpha \in \C$.
%\begin{equation*}
% a_m(\xi) = \sum_{\lvert \alpha \rvert_1 = m} a_\alpha \xi^\alpha, \quad \xi \in \R^d ,
%\end{equation*}
%the \emph{principal symbol} of $a$.
We say that the polynomial $a$ is \emph{strongly elliptic} if there exist constants $c > 0$ and $\omega  \in \R$ such
that $a$ satisfies for all $\xi \in \R^d$ the lower bound
\begin{equation}
\label{eq:strongly_elliptic}
\re a (\xi) \geq c\abs{\xi}^m - \omega .
\end{equation}
Note that strong ellipticity implies that $m$ is even.

Given a strongly elliptic polynomial $a$ and $p \in [1,\infty]$, we define the associated heat semigroup $S : [0,\infty) \to \mathcal{L} (L_p (\R^d))$ by
\begin{equation} \label{eq:semigroup}
 S_tf = \F^{-1}\euler^{-ta}\F f = \F^{-1}\euler^{-ta}*f .
\end{equation}
Note that the second equality holds since $\euler^{-ta} \in \mathcal{S}(\R^d)$. It is well known that the operator semigroup $(S_t)_{t\geq 0}$ is strongly continuous if $p \in [1,\infty)$. For $p=\infty$ the semigroup is the dual semigroup of a strongly continuous semigroup on $L_1(\R^d)$ and hence it is only weak*-continuous in general. For details we refer, e.g., to \cite{Arendt-02}.
By \cite{TerElstR-96}, the integral kernel $k_t = \mathcal{F}^{-1} \euler^{-ta}$ satisfies the following heat kernel estimate: There exist $c_1,c_2>0$ such that for all $x\in \R^d$ and $t>0$ we have 
\begin{equation}\label{eq:kernelbound}
\lvert k_t(x)\rvert \leq c_1 \euler^{\omega t} t^{-d/m} \euler^{-c_2\left(\frac{\lvert x \rvert^m}{t}\right)^{\frac{1}{m-1}}}.
\end{equation}
This implies that there is $M \geq 1$ and $\omega \in \R$ such that for all $p\in [1,\infty]$, $f\in L_p(\R^d)$, and $t \geq 0$ we have 
\begin{equation} \label{eq:realpart}
\lVert S_t f\rVert_{L_p(\R^d)} \le \lVert k_t \rVert_{L_1(\R^d)} \lVert f \rVert_{L_p(\R^d)} \leq M \euler^{\omega t} \lVert f \rVert_{L_p(\R^d)}.
\end{equation}
In order to formulate our main result, we introduce the notion of a thick subset $\thickset$ of $\R^d$.

\begin{Definition}
Let $\rho\in (0,1]$ and $L\in (0,\infty)^d$. A set $\thickset \subset \R^d$ is called \emph{$(\rho,L)$-thick} if $\thickset$ is measurable and for all $x \in \R^d$ we have
  \[
  \left\lvert \thickset \cap \left( \bigtimes_{i=1}^d (0,L_i) + x \right) \right\rvert \geq \rho \prod_{i=1}^d L_i .
  \]
Here, $\lvert \cdot \rvert$ denotes Lebesgue measure in $\R^d$.
\end{Definition}
%Given a thick set $\thickset \subset \R^d$ and $q \in [1,\infty]$, we define the operator $\mathbf{1}_\omega\from L_q(\R^d) \to L_q(\omega)$ by $\mathbf{1}_\omega f = f|_{\omega}$.

The following theorem yields a final-state observability estimate for $(S_t)_{t\geq0}$ on thick sets.

\begin{Theorem}\label{thm:obs} 
Let $m \in \N$, $a\from \R^d \to \C$ a strongly elliptic polynomial of order $m$, $c>0$ and $\omega\in\R$ as in \eqref{eq:strongly_elliptic}, and $(S_t)_{t\geq0}$ as in \eqref{eq:semigroup}. Let $\rho\in(0,1]$, $L \in (0,\infty)^d$,
$\thickset \subset \R^d$ a $(\rho,L)$-thick set, $p,r \in [1,\infty]$, and $T>0$. 
Then we have for all $f \in L_p(\R^d)$
\begin{equation*}
	\norm{S_T f}_{L_p(\R^d)} \leq \begin{cases}
    \displaystyle{C_{\mathrm{obs}} \left(\int_0^T \norm{(S_t f)|_\thickset}_{L_p(\thickset)}^r \drm t\right)^{1/r}} & \text{if } r\in [1,\infty),\\
    \displaystyle{ C_{\mathrm{obs}}\esssup\limits_{t\in [0,T]} \norm{(S_t f)|_\thickset}_{L_p(\thickset)}} & \text{if } r=\infty,
 \end{cases} 
\end{equation*}
where 
\[
C_{\mathrm{obs}} 
 =
 \frac{K_a}{T^{1/r}} \left( \frac{K_d}{\rho} \right)^{K_d(1+\lvert L \rvert_1 \lambda^*)}  \exp \left(\frac{K_m (\lvert L \rvert_1 \ln (K_d / \rho))^{m/(m-1)}}{(cT)^{1 / (m - 1)}} +   K\max\{\omega,0\} T\right).
\]
Here,  $\lambda^* = (2^{m+3} \max \{\omega , 0\} / c)^{1/m}$, $K>0$ is an absolute constant, and $K_a, K_d, K_m > 0$ are constants depending only on the polynomial $a$, on $d$, or on $m$, respectively.
\end{Theorem}
 \begin{Remark}[Optimality of $C_{\mathrm{obs}}$]
  In this remark we discuss the optimality of $C_{\mathrm{obs}}$ with respect to the parameters $T > 0$ and $L \in (0,\infty)^d$. 
%\paragraph{Optimality with respect to $T$}
\par
  We begin with optimality with respect to the time parameter $T$. Assume that $E \subset \R^d$ is a thick set, $r=2$, and the generator $-A_2$ of the semigroup $(S_t)_{t \geq 0}$ is a self-adjoint operator in $L_2 (\R^d)$ with $\min \sigma (A_2) = 0$. This is, e.g., the case if the strongly elliptic polynomial is given by $a (\xi) = \lvert \xi \rvert^2$. Then, the generator $-A_2$ of the semigroup $(S_t)_{t \geq 0}$ is the (self-adjoint) Laplacian on $L_2 (\R^d)$.
  Let us define the optimal observability constant by
  \[
  C_{\mathrm{obs}}^* (T) := \sup_{\genfrac{}{}{0pt}{1}{f \in L_2 (\R^d)}{f \not = 0}} \frac{\lVert S_T f \rVert_{L_2 (\R^d)}}{\left( \int_0^T \lVert S_t f|_E \rVert^2_{L_2 (E)} \drm t \right)^{1/2}} .
  \]
  An upper bound on $C_{\mathrm{obs}}^* (T)$ is given in Theorem~\ref{thm:obs}.
  As $\min \sigma (A_2) = 0$, for $\varepsilon>0$ we can choose a function $0\neq f_\varepsilon \in L_2(\R^d)$ from the spectral subspace of the interval $[0,\varepsilon]$, and calculate, using spectral calculus,
  \[
   C_{\mathrm{obs}}^* (T) 
   \geq 
   \frac{\euler^{-\varepsilon T} \lVert f_\varepsilon \rVert_{L_2 (\R^d)}}{\left( \int_0^T \lVert f_\varepsilon \rVert_{L_2 (\R^d)}^2 \drm t \right)^{1/2}} = \frac{\euler^{-\varepsilon T}}{T^{1/2}} .
  \]
  As $\varepsilon > 0$ was arbitrary, our $C_{\mathrm{obs}}$ in Theorem~\ref{thm:obs} is optimal with respect to $T$ for large $T$. We refer to \cite[Theorem~2.13]{NakicTTV-20} for a more general statement.
 In order to see that our bound is optimal for small $T$ as well, we might argue as follows. If one considers a linear control problem in $L_2 (\Omega)$ with bounded open $\Omega\subset\R^d$ instead of $\R^d$ and with $(S_t)_{t \geq 0}$ being the heat-semigroup, it has been shown that
\begin{equation} \label{eq:Zuazua-lower}
 \sup_{\overline{B (\rho)} \subset \Omega \setminus E} \rho^2 / 4 \leq \liminf_{T \to 0} T \ln C_{\mathrm{obs}}^* (T)   ,
\end{equation}
see \cite{Fernandez-CaraZ-00,Zuazua-01,Miller-04}. This shows that the exponential blowup for $T \to 0$ has to occur for the controlled heat equation on bounded open subsets $\Omega\subset\R^d$. In order to extend \eqref{eq:Zuazua-lower} to the case $\Omega = \R^d$ it seems feasible to apply the method obtained in \cite{SeelmannV-20}. In that paper, the authors show that if the controlled heat equation on $\Omega = \Lambda_L = (-L/2 , L/2)^d$ satisfies an observability estimate with a constant independent of $L>0$, then, using a continuity argument, the corresponding system on $\Omega = \R^d$ satisfies an observability estimate as well with the same upper bound. By an analogous argument, the lower bound \eqref{eq:Zuazua-lower} should hold in the case $\Omega = \R^d$ as well. This suggests that $C_{\mathrm{obs}}$ in Theorem~\ref{thm:obs} is optimal also in the regime $T \to 0$.
%
%\paragraph{Optimality with respect to $\rho$ and $L$}
\par
Let us now turn to a discussion of the optimality with respect to $L\in (0,\infty)^d$. We consider the case $r=p=2$ and $(S_t)_{t \geq 0}$ being the heat-semigroup. As above, we assume that the lower bound \eqref{eq:Zuazua-lower} holds in the case $\Omega = \R^d$. Then Theorem~\ref{thm:obs} implies (using $c = 1$ and $\omega = 0$)
\begin{equation} \label{eq:geometry1}
 \liminf_{T \to 0} T \ln C_{\mathrm{obs}}  =K_m \bigl(\lvert L \rvert_1 \ln (K_d / \rho) \bigr)^2
\end{equation}
Now we consider an example from \cite[Remark~4.14]{NakicTTV-20}. Let  $d = 2$, $l > 0$ and
\[
 E = \bigcup_{k \in l\Z} \left( (k-l/2 , k) \times \R \right) .
\]
The set $E$ is $(1/2 , L)$-thick with $L = (l,l)$ and the left hand side of \eqref{eq:Zuazua-lower} is given by
\begin{equation} \label{eq:geometry2}
  \sup_{\overline{B (\rho)} \subset \R^2 \setminus E} \rho^2 / 4 = \frac{l^2}{64} .
\end{equation}
From \eqref{eq:geometry1} and \eqref{eq:geometry2} we conclude  that $C_{\mathrm{obs}}$ in Theorem~\ref{thm:obs} is optimal with respect to $L$.
 \end{Remark}

%By duality, we thus obtain (cost-uniform approximate) null-controllability for \eqref{eq:system-intro}.
%\begin{Remark}[Discussion on observability and null-controllability]
 For $p\in [1,\infty)$ let $-A_p$ be the generator of the $C_0$-semigroup $(S_t)_{t\geq 0}$ on $L_p (\R^d)$. Note that for all $f\in \mathcal{S}(\R^d)$ we have
\[
A_p f = \sum _{\abs{\alpha}_1 \leq m} a_\alpha (-\ii)^{|\alpha|} \partial^\alpha f  .
\]
Moreover, by \eqref{eq:strongly_elliptic}, the differential operator $A_p$ is strongly elliptic, i.e., there is $c>0$ such that 
$$\re\Big(\sum _{\abs{\alpha}_1 = m} a_\alpha\xi^\alpha\Big) \geq c |\xi|^m.$$
Then, the statement of Theorem~\ref{thm:obs} corresponds to a final-state observability estimate for the system 
\begin{equation*}% \label{eq:system:obs}
\begin{aligned}
  \dot{x}(t) & = -A_p x(t), \quad & &t\in (0,T] ,\quad x(0)  = x_0 \in L_p(\R^d), \\
  y(t)  &= x(t)|_\thickset, \quad   & & t\in [0,T].
  \end{aligned}
\end{equation*}
\par
Let us now turn to the discussion on null-controllability. For a measurable set $\thickset\subset \R^d$ and $T>0$ we consider the linear control problem 
\begin{align*}
  \dot{x}(t) & = -A_p x(t) + \1_\thickset u(t), \quad t \in (0,T], \quad x(0) = x_0 \in L_p(\R^d)
	%\label{eq:system:con}
\end{align*}
where $u\in L_r ((0,T);L_p(\thickset))$ with $r\in[1,\infty]$. The unique mild solution
%of \eqref{eq:system:con} 
is given by Duhamel's formula
\begin{equation} \label{eq:Duhamel}
x(t) = S_t x_0 + \mathcal{B}^t u, \quad\text{where}\quad \mathcal{B}^t u = \int_0^t S_{t-\tau} \1_\thickset u(\tau) \drm\tau .
\end{equation}
By Theorem~\ref{thm:obs} and duality, we obtain (cost-uniform approximate) null-controllability for \eqref{eq:system-intro}.
%\end{Remark}
%
\begin{Theorem} \label{thm:null-control}
 Let $m \in \N$, $a\from \R^d \to \C$ a strongly elliptic polynomial of order $m$, $c>0$ and $\omega\in\R$ as in \eqref{eq:strongly_elliptic}, and $-A_p$ the generator of the $C_0$-semigroup $(S_t)_{t\geq0}$ as in \eqref{eq:semigroup}. Let $\rho\in(0,1]$, $L \in (0,\infty)^d$, $\thickset \subset \R^d$ a $(\rho,L)$-thick set, $r \in [1,\infty]$, and $T>0$. 
\begin{enumerate}[(a)]
 \item For any $x_0 \in L_1 (\R^d)$ and any $\varepsilon > 0$ there exists $u \in L_r ((0,T);L_1 (E))$ with
$$\|u\|_{L_r((0,T);L_1(\thickset))} \le C_{\mathrm{obs}}\|x_0\|_{L_1(\R^d)} \quad \text{and} \quad \lVert x(T) \rVert_{L_1(\R^d)} < \varepsilon,$$
where $x$ is the solution of \eqref{eq:system-intro} given by \eqref{eq:Duhamel}.
\item Let $p \in (1,\infty)$. Then for any $x_0 \in L_p (\R^d)$ there exists $u \in L_r ((0,T);L_p (E))$ with 
$$\|u\|_{L_r((0,T);L_p(\thickset))} \le C_{\mathrm{obs}}\|x_0\|_{L_p(\R^d)} \quad \text{and} \quad x(T) =0,$$
where $x$ is the solution of \eqref{eq:system-intro} given by \eqref{eq:Duhamel}.
\end{enumerate}
Here, $C_{\mathrm{obs}}$ is as in Theorem~\ref{thm:obs} with $r$ replaced by $r'$ where $r' \in [1,\infty]$ such that $1/r + 1/r' = 1$.
\end{Theorem}
The statement (a) of Theorem~\ref{thm:null-control} corresponds to \emph{cost-uniform approximate null-controllability in time $T$}, whereas the statement (b) corresponds to \emph{null-controllability in time $T$}.
Note that in case $p\in (1,\infty)$ null-controllability and cost-uniform approximate null-controllability are equivalent, see, e.g.,\ \cite{Carja-88}.

It is a standard duality argument that Theorem~\ref{thm:null-control} follows from Theorem~\ref{thm:obs} by means of Douglas' lemma.
For the sake of completeness, we give a short proof.
\begin{proof}[Proof of Theorem~\ref{thm:null-control} (assuming Theorem~\ref{thm:obs})]
 Let $p \in [1,\infty)$ and $r \in [1,\infty]$. Moreover, let $\mathcal{B}^T \colon \allowbreak L_r ((0,T); L_p (E)) \allowbreak \to L_p (\R^d)$ be given by
\begin{equation*} %\label{eq:controllability-map}
 \mathcal{B}^T u = \int_0^T S_{T-t} \1_E u (t)  \drm\tau .
\end{equation*}
Then, by \cite[Theorem 2.1]{Vieru-05} we have for all $g \in L_{p'}(\R^d)$
\[
 \lVert (\mathcal{B}^T)'g\rVert_{L_{r} ((0,T); L_{p} (E))'} 
 =
 \sup_{\tau\in [0,T]} \lVert (S'_{T-\tau} g) |_E\rVert_{L_{p'} (E)} = \sup_{t\in [0,T]} \lVert (S'_t g)|_E \rVert_{L_{p'} (E)}
 \]
 if $r = 1$, and
 \[
  \lVert (\mathcal{B}^T)'g \rVert_{L_{r} ((0,T); L_{p} (E))'} 
 =
 \left( \int_0^T \lVert (S'_{t-\tau} g)|_E \rVert_{L_{p'} (E)}^{r'} \drm \tau \right)^{1/{r'}} 
 = 
 \left( \int_0^T \lVert (S'_t g)|_E \rVert_{L_{p'} (E)}^{r'} \drm t\right)^{1/{r'}}
 \]
 if $r \in (1,\infty]$, where $r' \in [1,\infty]$ is such that $1 / r  + 1/ r' = 1$ and $p' \in (1,\infty]$ is such that $1 / p  + 1/ p' = 1$. Since $\mathcal{F}S_t' = \euler^{-ta(-\cdot)}\mathcal{F}$, we have that $(S_t')_{t\geq0}$ is associated to the symbol $a(-\cdot)$ which is strongly elliptic with the same constant $c>0$. Moreover, since the associated heat kernel is given by $(\mathcal{F}^{-1}\euler^{-ta})(-\cdot)$, we have $\norm{S_t'} \leq M\euler^{\omega t}$ with the same $M$ and $\omega$ as in \eqref{eq:realpart}. Thus, Theorem~\ref{thm:obs} and the above equalities imply for all $g \in L_{p'} (\R^d)$
 \begin{equation*}% \label{eq:Douglas-e}
  \lVert S_T' g \rVert_{L_{p'}} \leq C_{\mathrm{obs}} \lVert (\mathcal{B}^T)' g \rVert_{L_{r'} ((0,T);L_{p'} (E))} = C_{\mathrm{obs}} \lVert (\mathcal{B}^T)' g \rVert_{L_{r} ((0,T);L_{p} (E))'},
 \end{equation*}
 where $C_{\mathrm{obs}}$ is as in Theorem~\ref{thm:obs} with $r$ replaced by $r'$.
 By Douglas' lemma, see e.g.\ \cite{Harte-78,Carja-85,Carja-88}, we conclude
 %from \eqref{eq:Douglas-e} that
 \[
  \{S_T x_0 \colon \lVert x_0 \rVert_{L_p (\R^d)} \leq 1\}  
  \subset
  \overline{\{ \mathcal{B}^T u \colon \lVert u \rVert_{L_r ((0,T);L_p (E))} \leq C_{\mathrm{obs}} \}}
  \quad \text{if} \quad p = 1
 \]
 and
\[
  \{S_T x_0 \colon \lVert x_0 \rVert_{L_p (\R^d)} \leq 1\}  
  \subset
  \{ \mathcal{B}^T u \colon \lVert u \rVert_{L_r ((0,T);L_p (E))} \leq C_{\mathrm{obs}} \}
  \quad \text{if} \quad p \in (1,\infty) .
 \]
By scaling and linearity, this implies the statement of the theorem. 
\end{proof}
\section{Proof of Theorem~\ref{thm:obs}}
\label{sec:dissipation}

For the proof of Theorem~\ref{thm:obs} we apply the abstract observability estimate in Theorem~\ref{thm:spectral+diss-obs}. For this purpose, we define a familiy of operators $P_\lambda$, and verify the uncertainty principle \eqref{eq:ass:uncertainty} and the dissipation estimate \eqref{eq:ass:dissipation}.

We start with defining the operators $P_\lambda$ as in \cite{GallaunST-20}.
Let $\eta\in C_{\mathrm c}^\infty ([0,\infty) )$ with $0\leq\eta\leq 1$ such that $\eta (r) = 1$ for $r\in [0,1/2]$ and $\eta (r) = 0$ for $r\geq 1$. 
For $\lambda > 0$ we define $\chi_\lambda\from \R^d\to \R$ by $\chi_\lambda (\xi) = \eta (\lvert \xi \rvert / \lambda)$. Since $\chi_\lambda \in \mathcal{S}(\R^d)$, we have $\mathcal{F}^{-1}\chi_\lambda \in \mathcal{S}(\R^d) \subset L_1(\R^d)$ and for all $p \in [1,\infty]$ we define $P_\lambda \from  L_p(\R^d) \to L_p(\R^d)$ by $P_\lambda f = (\mathcal{F}^{-1}\chi_\lambda) * f$.
By Young's inequality we have for all $f\in L_p(\R^d)$
\[
\lVert P_\lambda f \rVert_{L_p(\R^d)} = \lVert (\mathcal{F}^{-1} \chi_\lambda) \ast f \rVert_{L_p(\R^d)} \leq \lVert \mathcal{F}^{-1} \chi_\lambda \rVert_{L_1(\R^d)} \lVert f \rVert_{L_p(\R^d)}.
\]
Moreover, the norm $\lVert \mathcal{F}^{-1} \chi_\lambda \rVert_{L_1(\R^d)}$ is independent of $\lambda >0$. Indeed, by the scaling property of the Fourier transform and by change of variables we have for all $\lambda > 0$
\begin{align} \label{eq:IndependedOfLambda}
\lVert \mathcal{F}^{-1} \chi_\lambda \rVert_{L_1(\R^d)} 
= 
\lvert \lambda \rvert^d \lVert (\mathcal{F}^{-1} \chi_1) (\lambda \cdot)\rVert_{L_1(\R^d)}
= 
\lVert \mathcal{F}^{-1} \chi_1 \rVert_{L_1(\R^d)}.
\end{align}

Next, we observe that the uncertainty principle \eqref{eq:ass:uncertainty} is a consequence of the following Logvinenko--Sereda theorem from \cite{Kovrijkine-01}, see also \cite{LogvinenkoS-74, Kovrijkine-00} for predecessors.
\begin{Theorem}[Logvinenko--Sereda theorem]
	\label{Thm:Logvinenko-Sereda_Rd}
	There exists $K\geq 1$ such that for all $p\in [1,\infty]$, $\lambda > 0$, $\rho \in (0,1]$, $L \in (0,\infty)^d$, $(\rho,L)$-thick sets $\thickset \subset \R^d$, and $f\in L_p(\R^d)$ satisfying $\supp \F f \subset [-\lambda,\lambda]^d$ we have
	\[
	\norm{f}_{L_p(\R^d)} \leq  d_0 \euler^{d_1\lambda}  \norm{f}_{L_p(\thickset)} ,
	\]
	where
	\begin{equation} \label{eq:d0d1}
	d_0 = \euler^{K d \ln (K^d / \rho)}
	\quad\text{and}\quad
	d_1 = 2 \lvert L \rvert_1 \ln (K^d / \rho) .
	\end{equation}
\end{Theorem}

Concerning the dissipation estimate \eqref{eq:ass:dissipation}, we prove the following Proposition.
\begin{Proposition} \label{Prop:Dissipation}
	Let $m \in \N$, $a\from \R^d \to \C$ a strongly elliptic polynomial of order $m$, $c>0$ and $\omega\in\R$ as in \eqref{eq:strongly_elliptic}, $(S_t)_{t\geq0}$ as in \eqref{eq:semigroup}, and $(P_\lambda)_{\lambda>0}$ as above. Then for all $p\in[1,\infty]$, $f\in L_p(\R^d)$, $\lambda>2^{(m+3)/m}(\max\{\omega,0\}/c)^{1/m}$, and $t\geq 0$ we have
	\begin{equation*} %\label{eq:dissipation}
		%\lVert S_t (\operatorname{Id} - P_\lambda) f \rVert_{L_p (\R^d)} \leq C \euler^{-\omega t\lambda^m} \lVert f \rVert_{L_p (\R^d)}
		%\quad
		\lVert (\operatorname{Id} - P_\lambda) S_t f \rVert_{L_p (\R^d)} \leq K_a \euler^{-2^{-m-3}c t\lambda^m} \lVert f \rVert_{L_p (\R^d)},
	\end{equation*}
	where $K_{a} \geq 0$ is a constant depending only on $a$ (and therefore also on $m$ and $d$).
\end{Proposition}

\begin{proof}
	For all $f\in L_p(\R^d)$, $\lambda>0$ and $t\geq 0$ we have
	\begin{equation*}
	(\operatorname{Id} - P_\lambda) S_t f =  \F^{-1}((1 - \chi_\lambda)\euler^{-ta})\ast f ,
	\end{equation*}
	and by Young's inequality, we thus obtain
	\begin{equation*}
	\lVert (\operatorname{Id} - P_\lambda) S_t f \rVert_{L_p (\R^d)} \le  \lVert  \F^{-1}((1 - \chi_\lambda) \euler^{-ta}) \rVert_{L_1 (\R^d)} \lVert f \rVert_{L_p (\R^d)}.
	\end{equation*}
	We write
	\begin{equation*}
		k_{t,\lambda} = \F^{-1}((1 - \chi_\lambda) \euler^{-ta}).
	\end{equation*}
	By Young's inequality, \eqref{eq:IndependedOfLambda} and \eqref{eq:kernelbound}, there exists $K_a\geq 0$ such that
	\begin{align*}
	\lVert \F^{-1}(\chi_\lambda \euler^{-ta}) \rVert_{L_1 (\R^d)}
	&= \lVert \mathcal{F}^{-1} \chi_\lambda \ast \F^{-1} \euler^{-ta} \rVert_{L_1 (\R^d)}
	\leq
	\norm{\F^{-1}\chi_1}_{L_1(\R^d)}\norm{\F^{-1}e^{-ta}}_{L_1(\R^d)}\\
	&\leq K_ae^{\omega t} .
	\end{align*}
	Hence we find for all $\lambda>0$ and $t \geq 0$ the bound
	\begin{equation} \label{eq:kmu-bounded}
	\lVert k_{t,\lambda} \rVert_{L_1 (\R^d)} \leq \lVert  \mathcal{F}^{-1} \euler^{-ta}  \rVert_{L_1 (\R^d)} + \lVert  \mathcal{F}^{-1} (\chi_{\lambda} \euler^{-ta})  \rVert_{L_1 (\R^d)} \leq K_a\euler^{\omega t}.
	%\leq (1+C) \lVert \mathcal{F}^{-1} \euler^{-1/2 a_\pi} \rVert_{L_1 (\R^d)} < \infty.
	\end{equation}
	Setting
	\[\sigma_{t,\lambda} = ((1 - \chi_\lambda)\euler^{-ta})(t^{-1/m}\cdot) = (1-\chi_{t^{1/m}\lambda})\euler^{-ta(t^{-1/m}\cdot)} ,\]
	for $\lambda>0$ and $t>0$, by linear substitution $\eta = t^{1/m}\xi$ it follows that
	\begin{equation*}
	k_{t,\lambda}(x) = \frac{1}{(2\pi)^d}\int_{\R^d} \euler^{\ii t^{-1/m}x\cdot \eta}((1 - \chi_\lambda)\euler^{-ta})(t^{-1/m}\eta)t^{-d/m} \drm\eta = t^{-d/m}(\F^{-1}\sigma_{t,\lambda})(t^{-1/m}x) .
	\end{equation*}
	Therefore,
	\begin{equation*}
		\norm{k_{t,\lambda}}_{L_1(\R^d)} = \norm{\F^{-1}\sigma_{t,\lambda}}_{L_1(\R^d)} .
	\end{equation*}
	Let now $\lambda > \lambda_0 = 2^{(m+3)/m}(\max\{\omega,0\}/c)^{1/m}$ and $\alpha \in \N_0^d$ with $\lvert \alpha \rvert_1 \leq d+1$. By differentiation properties of the Fourier transform we have 
	\begin{equation}\label{eq:MultSatz}
	(-\ii)^{\abs{\alpha}} x^\alpha\F^{-1}\sigma_{t,\lambda}(x) = \F^{-1}(\partial^\alpha\sigma_{t,\lambda})(x),\quad x\in\R^d.
	\end{equation}
% 	\todo{Wozu die folgenden 2 Zeilen?}
% 	and hence for all $x \in \R^d$ 
% 	\begin{align} 
% 	\lvert x^\alpha k_{t,\lambda} (x) \rvert 
% 	&= 
% 	\biggl\lvert \frac{1}{(2\pi)^d}\int_{\R^d} \euler^{\ii x \cdot \xi} \partial^\alpha\sigma_{t,\lambda}(\xi)\drm \xi \biggr\rvert
% 	\leq 
% 	\frac{1}{(2\pi)^d}\int_{\R^d} \bigl\lvert \partial^\alpha\sigma_{t,\lambda}(\xi)  \bigr\rvert \drm \xi .
% 	\label{eq:MultSatz}
% 	\end{align}
	We apply the product rule and the triangle inequality to obtain
	\begin{align} \label{eq:product-rule}
	\bigl\lvert \partial^\alpha \sigma_{t,\lambda} \bigr\rvert 
	& \leq  
	\1_{\{\abs{\xi}\geq t^{1/m}\lambda/2\}}\abs{\partial^\alpha \euler^{-ta(t^{-1/m}\cdot)}} + \sum_{\genfrac{}{}{0pt}{2}{\beta \in \N_0^d}{\beta < \alpha}} 
	\binom{\alpha}{\beta}  \bigl\lvert \partial^{\alpha - \beta} (1 - \chi_{t^{1/m}\lambda}) \bigr\rvert \bigl\lvert \partial^\beta  \euler^{-ta(t^{-1/m}\cdot)} \bigr\rvert  .
	\end{align}
	By \eqref{eq:strongly_elliptic}, for $t\geq 0$ and $\xi\in\R^d$ with $\abs{\xi}\geq \lambda/2$ we observe
	\begin{equation*} %\label{eq:deriv_est1}
		\abs{\euler^{-ta(\xi)}} = \euler^{-t\re a(\xi)} \leq \euler^{\omega t}\euler^{-ct\abs{\xi}^m} \leq \euler^{\omega t}\euler^{-ct\lvert \xi \rvert^m/2} \euler^{-ct\lambda^m/2^{m+1}} .
	\end{equation*}
	Hence, by possibly increasing $K_a$, for all $\beta\in \N_0^d$ with $\beta \leq \alpha$ and all $\xi \in \R^d$ with $\abs{\xi} \geq t^{1/m}\lambda/2$ we obtain 
	\begin{align*}
	\bigl\lvert (\partial^\beta  \euler^{-ta(t^{-1/m}\cdot)})(\xi) \bigl\rvert &\leq K_a\bigl(1 +t^{(m-1)/m}\abs{\xi}^{m-1} \bigr)^{\abs{\beta}_1}  \abs{\euler^{-ta(t^{-1/m}\xi)}} \\ &\leq K_a\euler^{\omega t}\bigl(1 +t^{(m-1)/m}\abs{\xi}^{m-1} \bigr)^{\abs{\beta}_1} \euler^{-c\lvert \xi \rvert^m/2} \euler^{-ct\lambda^m/2^{m+1}}  \\
	&\leq K_a\euler^{\omega t}\bigl(1 +t^{(m-1)/m}\abs{\xi}^{m-1} \bigr)^{\abs{\beta}_1} \euler^{-c\lvert \xi \rvert^m/2} \euler^{-ct\lambda_0^m/2^{m+2}}
	\euler^{-ct\lambda^m/2^{m+2}}.
	\end{align*}
	Since, for all $\beta\in\N_0^d$ with $\beta \leq \alpha$ we have
	\begin{equation*}
		\sup_{t\geq0, \xi\in\R^d} \bigl(1 +t^{(m-1)/m}\abs{\xi}^{m-1}\bigr)^{\abs{\beta}_1}\euler^{-c\abs{\xi}^m/4} \euler^{- ct\lambda_0^m/2^{m+2}} <\infty,
	\end{equation*}
	we may again increase $K_a$ such that for all $\lambda> \lambda_0$, $t>0$,  $\beta\in\N_0^d$ with $\beta\leq \alpha$ and $\abs{\xi}\geq t^{1/m}\lambda/2$ we have
	\begin{equation}
     \label{eq:deriv_est2}
	\bigl\lvert (\partial^\beta  \euler^{-ta(t^{-1/m}\cdot)})(\xi) \bigl\rvert \leq K_a\euler^{\omega t}\euler^{-c\lvert \xi \rvert^m/4} \euler^{-ct\lambda^m/2^{m+2}}   .
	\end{equation}	
	For all $\beta\in\N_0^d$ with $\beta < \alpha$ we have for all $\lambda>0$, $t\geq 0$ and $\xi\in\R^d$ that %and $\sup_{\gamma \leq \alpha} \sup_\xi \lvert (D_\xi^\gamma  \chi_1) (\xi) \rvert < \infty$, we have for all $\beta \leq \alpha$
	\begin{align}\label{eq:derivative_chi}
	\bigl\lvert \partial^{\alpha - \beta} (1 - \chi_{t^{1/m}\lambda})(\xi) \bigr\rvert
	&\leq (t^{1/m}\lambda)^{-\lvert \alpha - \beta \rvert_1}  (\partial^{\alpha - \beta} \chi_1)(\xi/(t^{1/m}\lambda))\mathbf{1}_{\{t^{1/m}\lambda/2 \leq \abs{\xi} \leq t^{1/m}\lambda\}}(\xi) \nonumber\\
	&\leq C(t^{1/m}\lambda)^{-\lvert \alpha - \beta \rvert_1}\mathbf{1}_{\{t^{1/m}\lambda/2 \leq \abs{\xi} \leq t^{1/m}\lambda\}}(\xi),
	\end{align}
	where $C = \max_{\beta<\alpha} \norm{\partial^{\alpha-\beta} \chi_1}_\infty$.
	Thus, from \eqref{eq:product-rule}, \eqref{eq:deriv_est2} and \eqref{eq:derivative_chi} and the definition of $\lambda_0$, we may increase $K_a$ such that for all $\lambda>\lambda_0$ and $t>0$ such that $t^{1/m}\lambda\geq 1$ and all $\xi\in\R^d$ we have
% 	\begin{align*}
% 		\bigl\lvert \partial^\alpha [(1 - \chi_\lambda)\euler^{-ta} ]\bigr\rvert \leq& K_a \mathbf{1}_{\{\abs{\xi}\geq \lambda/2\}}\bigl\lvert \partial^\alpha \euler^{-ta} \bigr\rvert \\
% 		&+ K_a\mathbf{1}_{\{\lambda/2 \leq \abs{\xi} \leq 2\lambda\}}\sum_{\genfrac{}{}{0pt}{2}{\beta \in \N_0^d}{\beta < \alpha}}(t^{1/m}\lambda)^{-\lvert \alpha - \beta \rvert_1} \bigl\lvert \partial^\beta  \euler^{-ta}\bigr\rvert   .
% 	\end{align*}
% 	Let us assume for the moment that $t^{1/m}\lambda \geq 1$. From \eqref{eq:deriv_est1} and \eqref{eq:deriv_est2}, we may deduce in this case that
	\begin{equation*}
		\bigl\lvert \partial^\alpha\sigma_{t,\lambda}(\xi)  \bigr\rvert \leq K_a\euler^{\omega t}\euler^{-c\lvert \xi \rvert^m/4} \euler^{-ct\lambda^m/2^{m+2}} \leq K_a\euler^{-c\lvert \xi \rvert^m/4} \euler^{-ct\lambda^m/2^{m+3}} ,
	\end{equation*}
% 	By our choice of $\lambda_0$, we obtain
% 	\begin{equation*}
% 		\bigl\lvert D_\xi^\alpha\sigma_{t,\lambda}(\xi)  \bigr\rvert \leq K_a\euler^{-c\lvert \xi \rvert^m/4} \euler^{-ct\lambda^m/2^{m+3}} .
% 	\end{equation*}
	and hence, from \eqref{eq:MultSatz}, we may increase $K_a$ such that for all $x \in \R^d$ we obtain
	\begin{equation}\label{eq:x^alphakmu}
	\lvert x^\alpha \F^{-1}\sigma_{t,\lambda} (x) \rvert 
	\leq 
	K_a \euler^{- ct\lambda^m/ 2^{m+3}} .
	\end{equation}
	In particular, for $j \in \{1,2,\ldots , d\}$ and $\alpha_j = (d+1)e_j$, where $e_j$ denotes the $j$-th canonical unit vector in $\R^d$, we obtain $\lvert x_j \rvert^{d+1} \lvert \F^{-1}\sigma_{t,\lambda} (x) \rvert \leq K_a\euler^{- ct\lambda^m/ 2^{m+3}}$, hence $\lVert x \rVert_\infty^{d+1} \lvert \F^{-1}\sigma_{t,\lambda} (x) \rvert \leq K_a\euler^{- ct\lambda^m/ 2^{m+3}}$, and consequently for all $\lambda>\lambda_0$, $t>0$ such that $t^{1/m}\lambda \geq 1$ we find for all $x\in\R^d \setminus \{0\}$
	\begin{equation}\label{eq:consequence}
	\lvert\F^{-1}\sigma_{t,\lambda} (x) \rvert \leq K_a\euler^{- ct\lambda^m/2^{m+3}}\lvert x \rvert^{-d-1}  .
	\end{equation}
	From \eqref{eq:x^alphakmu} with $\alpha = 0$ and \eqref{eq:consequence} we obtain for all $\lambda>\lambda_0$, $t>0$ such that $t^{1/m}\lambda \geq 1$ that
	\begin{align*}
	\norm{k_{t,\lambda}}_{L_1(\R^d)} = \lVert \F^{-1}\sigma_{t,\lambda} \rVert_{L_1 (\R^d)} &\leq K_a \euler^{- ct\lambda^m/ 2^{m+3}} \Bigl(\int_{\abs{x}\leq 1}  \drm x +  \int_{\abs{x}>1}  \lvert x \rvert^{-d-1}  \drm x\Bigr) \\
	& \leq K_a \euler^{- ct\lambda^m/ 2^{m+3}},
	\end{align*}
	where we again increased $K_a$ in the last estimate.
	In view of \eqref{eq:kmu-bounded}, for $\lambda>\lambda_0$, $t>0$ such that $t^{1/m}\lambda \leq 1$ we have
	\[\norm{k_{t,\lambda}}_{L_1(\R^d)} \leq K_a \euler^{\omega t} \leq K_a\euler^{\omega/\lambda_0^m} \leq K_a\euler^{\omega/\lambda_0^m} \euler^{c/2^{m+3}} \euler^{-ct\lambda^m/2^{m+3}}.\]
	Hence, incresing $K_a$ again, we finally obtain for all $\lambda > \lambda_0$ and $t>0$ that
	\begin{equation*} %\label{eq:kmu-final}
	\lVert k_{t,\lambda} \rVert_{L_1 (\R^d)} \leq K_a \euler^{-ct\lambda^m /2^{m+3}} . \qedhere
	\end{equation*}
\end{proof}

We can finally prove Theorem~\ref{thm:obs}.

\begin{proof}[Proof of Theorem~\ref{thm:obs}]
	Let $(P_\lambda)_{\lambda > 0}$ be the family of operators defined at the beginning of this section. Then we have $\supp \F(P_\lambda f) \subset [-\lambda,\lambda]^d$ for all $\lambda > 0$ and all $f \in L_p(\R^d)$. Thus, Theorem~\ref{Thm:Logvinenko-Sereda_Rd} implies that for all $f \in L_p(\R^d)$ and all $\lambda > 0$ we have
	\begin{equation*}
	\norm{P_\lambda f}_{L_p(\R^d)} \leq  d_0 \euler^{d_1\lambda}\norm{P_\lambda f}_{L_p(\thickset)} ,
	\end{equation*}
	where $d_0$ and $d_1$ are as in \eqref{eq:d0d1}. Moreover, according to Proposition~\ref{Prop:Dissipation}, for all $\lambda > \lambda^*$ and all $f \in L_p(\R^d)$ we have
	\begin{equation*}
	\norm{(I - P_\lambda)S_t f}_{L_p(\R^d)} \leq d_2\euler^{-d_3\lambda^mt}\norm{f}_{L_p(\thickset)} ,
	\end{equation*}
	where $\lambda^* := (2^{m+3} \max \{\omega , 0\} / c)^{1/m}$, $d_2 \geq 1$ depends only on the polynomial $a$, and where $d_3 := 2^{-m-3} c$.
Moreover, the function $t\mapsto \norm{(S_t f)|_\thickset}_{L_p (\thickset)}$ is Borel-measurable for all $f\in L_p(\R^d)$. Indeed, if $p \in [1,\infty)$ the semigroup $(S_t)_{t \geq 0}$ is strongly continuous and the measurability follows. If $p = \infty$ the semigroup $(S_t)_{t \geq 0}$ is the dual of a strongly continuous semigroup on $L_1(\R^d)$. By means of the Hahn--Banach theorem the function $t\mapsto \norm{(S_t f)|_\thickset}_{L_\infty (\thickset)}$ is, as the supremum of continuous functions, lower semicontinuous and hence measurable. 
	Thus, we can apply Theorem~\ref{thm:spectral+diss-obs} with $X = L_p(\R^d)$, $Y = L_p(\thickset)$, $C\from X \to Y$ given by the restriction map on $\thickset$, and obtain that the statement of the theorem holds with $C_{\mathrm{obs}}$ replaced by
	\begin{align*}
	\tilde C_{\mathrm{obs}} := \frac{C_1}{T^{1/r}} \exp \left(\frac{C_2}{T^{\frac{1}{m - 1}}} +   C_3 T\right), 
	\end{align*}
	where $T^{1/r} = 1$ if $r=\infty$, and %(\max\{-\tilde \omega,\allowbreak 0\}  / c)^{1/m}
	\begin{align*}
	C_1 
	&:= (4 M d_0) \max \Bigl\{\left( 4d_2 M^2   (d_0 +1) \right)^{8/(\euler \ln 2)}, \euler^{4d_1 2\lambda^*}\Bigr\}, \\
	C_2 &:= 4 \bigl(2 \cdot 8^\frac{m}{m-1} d_1^{m} / d_3 \bigr)^{\frac{1}{m-1}} , \\[1ex]
	C_3 & := \max\{\omega , 0\} \bigl(1 + 10 / (\euler \ln 2) \bigr),
	\end{align*}
	with $M$ as in \eqref{eq:realpart}. We denote by $K_d$, $K_m$, and $K_a$ positive constants which depend only on the dimension $d$, on $m$, or on the polynomial $a$, respectively. A straightforward calculation shows that
	\begin{equation*}
	C_1 \leq K_a \left( \frac{K_d}{\rho} \right)^{K_d(1+\lvert L \rvert_1 \lambda^*)}
	\quad\text{and}\quad
	C_2 \leq \frac{K_m (\lvert L \rvert_1 \ln (K_d / \rho))^{m/(m-1)}}{c^{1/(m-1)}} .
	\end{equation*}
	Thus we obtain
	% \[
	%  \tilde C_{\mathrm{obs}} 
	%  \leq
	%  \frac{K_a}{T^{1/r}} \left( \frac{K_d}{\rho} \right)^{K_d(1+K_m\lvert L \rvert_1 (\max\{-\tilde \omega,\allowbreak 0\}  / c)^{1/m}))}  \exp \left(\frac{K_m (\lvert L \rvert_1 \ln (K_d / \rho))^{m/(m-1)}}{(cT)^{1 / (m - 1)}} +   C_3 T\right) =: C_{\mathrm{obs}} . \qedhere
	% \]
	\[
	\tilde C_{\mathrm{obs}} 
	\leq
	\frac{K_a}{T^{1/r}} \left( \frac{K_d}{\rho} \right)^{K_d(1+\lvert L \rvert_1 \lambda^*)}  \exp \left(\frac{K_m (\lvert L \rvert_1 \ln (K_d / \rho))^{m/(m-1)}}{(cT)^{1 / (m - 1)}} + C_3 T\right) =: C_{\mathrm{obs}} . \qedhere
	\]
\end{proof}

\section{Discussion on related questions and further research}
\label{sec:further_directions}

The focus of the whole paper is put more or less on observability and cost-uniform (approximate) null-controllability for linear control problems in Banach spaces. In particular, we study controllability to one single state at a given fixed time. In this section, we discuss related questions and possible generalizations, in particular, we address
%
% Up to now we have dealt with linear control and observability problems and for the control poblem we considered cost-uniform (approximate) null-controllability, i.e.\ controllability to one single state at a single time point. Clearly, some natural questions arise, namely:
\begin{enumerate}[(i)]
 \item control to a given trajectory defined on a time interval;
 \item possible generalization to non-linear problems;
 \item how to determine a control function.
\end{enumerate}
%We will comment on each point in a seperate subsection. 

\subsection{Control to trajectories}
 We consider parabolic control systems on $L_p(\R^d)$, $p\in [1,\infty)$, of the form
\begin{equation*}
  %\label{eq:system-intro}
  \dot{x}(t) = -A_p x(t) + \1_\thickset u(t),\quad t\in (0,T],\quad x(0) = x_0\in L_p(\R^d),
\end{equation*}
where $-A_p$ is a strongly elliptic differential operator of order $m \in \N$ with constant coefficients, $\1_\thickset \from L_p (\thickset)\to L_p(\R^d)$ is the embedding from a measurable set $\thickset \subset \R^d$ to $\R^d$,  $T>0$, and where $u \in L_r ((0,T);L_p (\thickset))$ with some $r \in (1,\infty)$. (For the sake of simplicity we restrict the discussion in this section to $r \in (1,\infty)$ only and exclude the cases $r=1$ and $r=\infty$.) Moreover, we introduce the cost functional $J : L_r ((0,T); L_p (E)) \to [0,\infty)$ by 
\[
 J (u) = \int_0^T \alpha (t) \lVert u(t) \rVert_{L_p (E)}^r \drm t + \int_{0}^T \beta (t) \lVert x (t) - x^{\mathrm{d}} (t) \rVert_{L_p (\R^d)}^r \drm t ,
\]
with some time-dependent weights $0 < \alpha_1 \leq \alpha(t) \leq \alpha_2 < \infty$ and $0 < \beta_1 \leq \beta (t) \leq \beta_2 < \infty$ for all $t \in [0,T]$, and some desired trajectory $x^{\mathrm{d}} : (0,T) \to L_p (\R^d)$. 
Thus, the first term describes the (weighted) cost of the control function, while the second penalizes deviations from a given trajectory.
Let $\varepsilon > 0$, $x_0 \in L_p (\R^d)$ be our initial state, and $x^* \in L_p (\R^d)$ a given target state. We are interested in the optimal control problem 
\begin{equation}\label{eq:problem-optimal-control}
 \min_{u \in L_r ((0,T);L_p (E))} \left\{ J (u) \colon \lVert x (T) - x^* \rVert_{L_p (\R^d)} \leq \varepsilon \right\} .
\end{equation}
That is, we are interested in finding a cost-optimal control function that steers our system up to an error $\varepsilon$ to a given target state. 
If $r,p \in (1,\infty)$, by standard convex optimization wisdom it follows that the problem \eqref{eq:problem-optimal-control} has a unique solution which we denote by $u^{\mathrm{opt}}$. Indeed, since the domain of $J$ is a reflexive space, and the functional $J$ is proper, convex, coercive, lower-semicontinuous and strictly convex in the sense of \cite[Theorem~2.19]{Peypouquet-15}, it follows that $\operatorname{argmin} J$ is a singleton. We denote the global minimizer of $J$ by $u^\mathrm{min}$. 
In order to see that the constrained minimization problem \eqref{eq:problem-optimal-control} has a unique solution as well, one introduces the indicator function 
\[
 I (x) = \begin{cases}
          0 & \text{if} \ \lVert x-x^* \rVert_{L_p (\R^d)} \leq \varepsilon ,\\
          +\infty  & \text{else} .
         \end{cases}
\]
Then, our problem \eqref{eq:problem-optimal-control} is equivalent to the problem
\[
 \min_{u \in L_r ((0,T);L_p (E))} \hat J (u) , 
 \quad\text{where}\quad
 \hat J (u) = J (u) + I (x(T)) .
\]
Since $\hat J$ is again proper, convex, coercive, lower-semicontinuous and strictly convex, \cite[Theorem~2.19]{Peypouquet-15} implies that $\operatorname{argmin} \hat J$ is a singleton. Note that the unique solution $u^{\mathrm{opt}}$ of \eqref{eq:problem-optimal-control} depends on the initial state $x_0$ and the target state $x^*$. Some remarks are in order:
\begin{enumerate}[(i)]
 \item If $\alpha = 1$, $\beta = 0$, and $x^* = 0$, then the latter minimization problem corresponds to the control problem studied in the paper under consideration. In particular, $u^{\text{opt}}$ corresponds to our control function, and $J (u^\text{opt})^{1/r}$ to the $x_0$-dependent control cost. In particular, in this paper we prove an upper bound on the control cost $J (u^\text{opt})^{1/r}$ which is uniform in the choice of the initial state $x_0 \in L_p (\R^d)$. Moreover, we give explicit dependence of the upper bound on the geometric properties of the control set $E$.
 \item If $\beta \not = 0$, then the second term in $J$ models controls to trajectories, as it punishes large deviations from the desired trajectory. Such problems have been already studied in the case of Hilbert spaces, i.e.\ where $r=p=2$, see \cite{LazarM-21} and the references therein. The paper \cite{LazarM-21} provides an abstract formula for $u^\text{opt}$, which is used to derive numerical algorithms for the approximation of $u^\text{opt}$ in certain situations.
\end{enumerate}
Let us finish this section by addressing some possible future projects.
One possible goal might be to prove an abstract formula for $u^{\mathrm{opt}}$ in the Banach space setting, just in the manner of, e.g., the results in \cite{LazarM-21}. Such an abstract formula can then be applied to numerical studies. Another goal might be to generalize our upper bound on the control cost to the setting of control to trajectories, i.e. providing an upper bound on $J (u^\text{opt})$.
The first question appears to be well worth studying: However, it seems that the methods to answer this question are entirely different from those used for our results. The second question is to us at least as interesting, but we believe it to be more challenging. Recall that in the case $\beta = 0$ we derive our result by giving an observability estimate. By duality such an observability constant gives an upper bound on $J(u_{\text{opt}})$.
To the best of our knowledge, if $\beta \neq 0$ this duality argument, and hence our whole approach, fails. This means it is not clear how to easily employ our main result on an observability estimate to control to trajectories.
\subsection{Controllability of non-linear problems}
A general strategy to treat non-linear problems stems from linearisations (i.e.\ first considering linear problems) and fixed point arguments under suitable assumptions on the non-linearity (such as small Lipschitz constants).
Let us restrict to semilinear problems of the form
\begin{equation}
  \label{eq:semilinear}
  \dot{x}(t) = -A_p x(t) + f(x(t),\nabla x(t)) + \1_\thickset u(t),\quad t\in (0,T],\quad x(0) = x_0\in L_p(\R^d),
\end{equation}
where $f\from\R\times \R^d\to\R$ is smooth and locally Lipschitz with $f(0,0) = 0$ and a suitable growth condition at infinity.
For the case of the semilinear heat equation (sometimes with no gradient dependence of $f$), where $-A_p = \Delta$ is the Laplacian, cost-uniform null-controllability was studied both in bounded \cite{FabrePZ-95,Fernandez-Cara-97,Fernandez-CaraZ-00} as well as unbounded domains \cite{Zuazua-01,CabanillasMZ-01,Zuazua-07, Gonzalez-BurgosT-07}. Since we are interested in unbounded domains, we will focus on this case here. There, a typical condition needed is that $\R^d\setminus E$ is bounded. Note that in the linear case we only require that $E$ is thick, which is much weaker. 

Let us sketch the method in the case of the semilinear heat equation. Controllability of \eqref{eq:semilinear} is derived in two steps. First, one proves an observability estimate for the dual system of the linearised system, i.e.\ of a system of the form
\begin{align*} 
 \dot{x}(t) & = -A_p' x(t) + a x(t) - \operatorname{div}(bx(t)), \quad t\in (0,T], \quad x(0) = x_0\in L_{p'}(\R^d)\\
 y(t) & = x(t)|_E, \quad t\in (0,T].
\end{align*}
Here, $a$ and $b$ come from linearising $f$. This yields cost-uniform null-controllability of the linearised system
\[\dot{x}(t) = -A_p x(t) + a x(t) + b \cdot \nabla x(t) + \1_E u(t) , \quad t\in (0,T],\quad x(0) = x_0\in L_p(\R^d).\]

For the second step, we note that for controllability of \eqref{eq:semilinear} it suffices to obtain controllability of 
\[\dot{x}(t) = -A_p x(t) + \1_{\R^d\setminus E} f(x(t),\nabla x(t)) + \1_\thickset u(t),\quad t\in (0,T],\quad x(0) = x_0\in L_p(\R^d);\]
cf.\ e.g.\ \cite[Section 3]{CabanillasMZ-01}.
Writing $f(x, \nabla x) = F(x) x + G(x) \cdot \nabla x$, for fixed $\tilde{x}$ we get the linearised control problem with $a = F(\tilde{x})$ and $b = G(\tilde{x})$. One then establishes a map $N\from \tilde{x}\mapsto x$, which maps $\tilde{x}$ to the solution state $x$ of the linearised control problem. The goal is then to show that $N$ has a fixed point, which may be done by applying Schauder's fixed point theorem.

It seems an interesting open problem whether our result for the linear case can be extended to the semilinear situation as sketched above.

\subsection{Determining control functions}

Our result states (approximate) null-controllability for \eqref{eq:system-intro}, i.e.\ for all $x_0$ there exists $u$ such that $x(T) = 0$ (has arbitrarily small norm). Thus, we obtain an existence statement.

Of course, in applied contexts, one would not only want to know that such a control function $u$ exists but how one can construct it. For the Hilbert space case $p=r=2$, one can either employ that control functions can be considered as orthogonal projections on an affine subspace, see \cite[Remark 2.6]{NakicTTV-20}, or consult the practical algorithm obtained in \cite[Sections 4,5]{LazarM-21} to construct the optimal minimizer $u^\text{opt}$. This changes the perspective from (functional analytic) mathematical control theory to PDE-constrained optimization.
%, where also practical algorithms to construct the optimal minimizer $u^\text{opt}$ are considered. 
%Some results on constructing an algorithm for the Hilbert space case $p=r=2$ can be found e.g.\ in \cite[Sections 4,5]{LazarM-21}. 
Note that, to the best of our knowledge, some arguments used there require to work in Hilbert spaces, so they are not available in our general Banach space situation. 
%Since we do not intend to make this change of perspective, we will not pursue this question in the manuscript under consideration. 
Since this question is of interest nonetheless, it may be dealt with in a forthcoming paper.
\paragraph{Acknowledgement} We thank the anonymous referees for their careful proofreading, as well as for their valuable comments which helped to significantly improve the manuscript.

\appendix

\section{Sufficient criteria for observability in Banach spaces}

We provide an abstract sufficient criteria for final-state observability in Banach space, which is a slight generalization of Theorem 2.1 in \cite{GallaunST-20}, see also \cite{Miller-10,TenenbaumT-11,WangZ-17,BeauchardP-18,NakicTTV-20,Barcena-PetiscoZ-21} for earlier results. In particular, it does not assume strong continuity of the semigroup. For the proof, we comment only on the necessary modifications compared to \cite{GallaunST-20}.
\begin{Theorem} \label{thm:spectral+diss-obs} 
Let $X$ and $Y$ be Banach spaces, $C\colon X\to Y$ a bounded linear operator, $(S_t)_{t\geq 0}$ a semigroup on $X$, $M \geq 1$ and $\omega \in \R$ such that $\lVert S_t \rVert \leq M \euler^{\omega t}$ for all $t \geq 0$, 
and assume that for all $x\in X$ the mapping $t\mapsto \norm{C S_t x}_Y$ is measurable.
Further, let $\lambda^* \geq 0$, $(P_\lambda)_{\lambda>\lambda^*}$ a family of bounded linear operators in $X$,
$r \in [1,\infty]$, $d_0,d_1,d_3,\gamma_1,\gamma_2,\gamma_3,T > 0$ with $\gamma_1 < \gamma_2$, and $d_2\geq 1$, and assume that
\begin{align} 
\forall x\in X \ \forall \lambda > \lambda^* &\colon \quad \lVert P_\lambda x \rVert_{ X } \le d_0 \euler^{d_1 \lambda^{\gamma_1}} \lVert C  P_\lambda x \rVert_{Y }
\label{eq:ass:uncertainty} , 
\end{align}
and
\begin{align}
\forall x\in X \ \forall \lambda > \lambda^* \ \forall t\in (0,T/2] &\colon \quad \lVert (\id-P_\lambda) S_t x \rVert_{X} \le d_2 \euler^{-d_3 \lambda^{\gamma_2} t^{\gamma_3}} \lVert x \rVert_{X} \label{eq:ass:dissipation} .
\end{align}
Then we have for all $x \in X$
\begin{equation*} %\label{eq:obs}
 \lVert S_T x \rVert_{X} \leq \begin{cases}
    C_{\mathrm{obs}} \left(\int_0^T \norm{CS_t x}_Y^r \drm t\right)^{1/r} & \text{if } r\in [1,\infty),\\
    C_{\mathrm{obs}}\esssup_{t\in [0,T]} \norm{CS_t x}_Y & \text{if } r=\infty ,
 \end{cases} 
\end{equation*}
with
\begin{align*}
 C_{\mathrm{obs}} = \frac{C_1}{T^{1/r}} \exp \left(\frac{C_2}{T^{\frac{\gamma_1 \gamma_3}{\gamma_2 - \gamma_1}}} +   C_3 T\right), 
\end{align*}
where $T^{1/r} = 1$ if $r=\infty$, and
\begin{align*}
  C_1 
  &= (4 M d_0) \max \Bigl\{\left( 4d_2 M^2   (d_0 \lVert C \rVert+1) \right)^{8/(\euler \ln 2)}, \euler^{4d_1\left(2\lambda^*\right)^{\gamma_1}}\Bigr\}, \\
 C_2 &= 4 \bigl(2^{\gamma_1} (2\cdot 4^{\gamma_3})^\frac{\gamma_1 \gamma_2}{\gamma_2-\gamma_1} d_1^{\gamma_2} / d_3^{\gamma_1} \bigr)^{\frac{1}{\gamma_2-\gamma_1}} , \\[1ex]
 C_3 & = \max\{\omega , 0\} \bigl(1 + 10 / (\euler \ln 2) \bigr).
\end{align*}
\end{Theorem}

The assumption in \eqref{eq:ass:uncertainty} is an abstract uncertainty principle (sometimes also called spectral inequality), while \eqref{eq:ass:dissipation} is a dissipation estimate. Thus, Theorem \ref{thm:spectral+diss-obs} can be rephrased that an uncertainty principle together with a dissipation estimate implies a final-state observability estimate.

\begin{Remark} %\label{rem:spectral+diss-obs}
  In the situation of Theorem~\ref{thm:spectral+diss-obs}, if we assume that $t\mapsto CS_t x$ is Bochner measurable, we can rewrite the statement of the theorem as
\begin{align*}
 \lVert S_T x \rVert_{X} \leq C_{\mathrm{obs}} \lVert C S_{(\cdot)} x \rVert_{L_r((0,T);Y)}.
\end{align*}
\end{Remark}

\begin{proof}[Proof of Theorem~\ref{thm:spectral+diss-obs}]
Since we do not assume the semigroup $(S_t)_{t \geq 0}$ to be strongly continuous, we cannot apply \cite[Theorem 2.1]{GallaunST-20} directly. The strong continuity of $(S_t)_{t \geq 0}$ was assumed in \cite{GallaunST-20} in order to ensure that for all $x \in X$ and $\lambda > \lambda^*$ the functions
\begin{align*}
   F (t) &= \bigl\lVert S_t x \bigr\rVert_X ,  
 & F_\lambda (t) &= \bigl\lVert P_\lambda S_t x \bigr\rVert_X ,
 & F_\lambda^\perp (t) &= \bigl\lVert (\id - P_\lambda) S_t x \bigr\rVert_X , \\
   G (t) &= \bigl\lVert C S_t x \bigr\rVert_Y ,
 & G_\lambda (t) &= \bigl\lVert C P_\lambda S_t x \bigr\rVert_Y ,
 & G_\lambda^\perp (t) &=\bigl\lVert C (\id - P_\lambda) S_t  x \bigr\rVert_Y ,
 \end{align*}
 are measurable. The measurability of these six functions was used to obtain the estimate
 \begin{align*}
 F (t) &\leq
 \frac{2M\euler^{\omega_+ T} d_0 \euler^{d_1\lambda^{\gamma_1}}}{t} \int_{t/2}^t G(\tau) \drm \tau +  \frac{d_2 M^{2} \euler^{5\omega_+ T / 4} \euler^{d_1\lambda^{\gamma_1}}}{\euler^{d_3 \lambda^{\gamma_2} (t/4)^{\gamma_3}}}\left(  d_0 \lVert C \rVert +1 \right) F(t / 4) ,
\end{align*}
where $\omega_+ = \max \{0, \omega\}$.
Such an inequality implies the statement of the theorem by iteration, see \cite{GallaunST-20}.
Thus it suffices to show the last displayed inequality by assuming merely measurability of the mapping $t \mapsto G (t)$.
Let $t > 0$, $\tau \in [t/2 , t]$ and $x\in X$. Since $F (\tau) \leq F_\lambda (\tau) + F_\lambda^\perp (\tau)$, by our assumptions and by the semigroup property we obtain 
\begin{align*}
 F (\tau) & \leq d_0 \euler^{d_1 \lambda^{\gamma_1}} G_\lambda (\tau) + d_2 \euler^{- d_3 \lambda^{\gamma_2} (\tau/2)^{\gamma_3}} F (\tau / 2) .
\end{align*}
Using $G_\lambda (\tau) \leq G (\tau) + G_\lambda^\perp (\tau) \leq  G (\tau) + \lVert C \rVert F_\lambda^\perp (\tau)$, our assumption, $\euler^{d_1 \lambda^{\gamma_1}} \geq 1$, and $F (\tau / 2) \leq M \euler^{\omega_+ t / 4} F (t / 4)$ we obtain
\begin{align*}
F (\tau) \leq d_0 \euler^{d_1 \lambda^{\gamma_1}} G(\tau) +(d_0 \lVert C \rVert + 1) d_2 \euler^{- d_3 \lambda^{\gamma_2} (\tau/2)^{\gamma_3}} \euler^{d_1 \lambda^{\gamma_1}} M \euler^{\omega_+ t / 4} F (t / 4) .
\end{align*}
 Since $F (t) \leq M  \euler^{\omega_+ t} F (\tau) $, we obtain
\begin{align*}
 F (t)\leq
 M \euler^{\omega_+ t} d_0 \euler^{d_1 \lambda^{\gamma_1}} G(\tau) +(d_0 \lVert C \rVert + 1) d_2 \euler^{- d_3 \lambda^{\gamma_2} (\tau/2)^{\gamma_3}} \euler^{d_1 \lambda^{\gamma_1}} M^2 \euler^{\omega_+  5t / 4} F (t / 4)  .
\end{align*}
Since the mapping $\tau \mapsto G(\tau)$ is measurable by assumption, we can integrate this inequality with respect to $\tau$, and obtain the desired estimate.
\end{proof}
%
%
%  \bibliographystyle{alpha}
%  \bibliography{../short3}

%
\end{document}